\documentclass[12pt,reqno]{article}

\usepackage{enumitem}
\usepackage{setspace}
\usepackage{amsmath}
\usepackage{amsfonts}
\usepackage{amssymb}
\usepackage{wasysym}
\usepackage{amsthm}
\usepackage{mathrsfs}
\usepackage{amscd}
\usepackage{mathtools}
\usepackage{tabu}
\usepackage{tikz}
\usepackage{tikz-cd}
\usepackage{hyperref}
\usepackage[margin=1in]{geometry}
\usepackage[matrix,arrow,curve,frame]{xy}
\usepackage{xy}
\usepackage[capitalise,noabbrev]{cleveref}
\usepackage{textcomp}

\definecolor{my-linkcolor}{rgb}{0.75,0,0}
\definecolor{my-citecolor}{rgb}{0.1,0.57,0}
\definecolor{my-urlcolor}{rgb}{0,0,0.75}
\hypersetup{
	colorlinks, 
	linkcolor={my-linkcolor},
	citecolor={my-citecolor}, 
	urlcolor={my-urlcolor}
}

\title{Coset Vertex Algebras}
\author{Thomas Creutzig and Andrew R. Linshaw}
\date{}

\numberwithin{equation}{section}

\theoremstyle{definition}\newtheorem{rema}{Remark}[section]
\theoremstyle{plain}
\newtheorem{theo}[rema]{Theorem}
\theoremstyle{definition}
\theoremstyle{plain}\newtheorem{lemma}[rema]{Lemma}
\newtheorem{corol}[rema]{Corollary}
\newtheorem{conj}[rema]{Conjecture}
\theoremstyle{definition}
\theoremstyle{definition}
\theoremstyle{definition}
    
\theoremstyle{definition}

\newcommand{\cA}{\mathcal{A}}

\newcommand{\cH}{\mathcal{H}}

\newcommand{\cC}{\mathcal{C}}

\newcommand{\cW}{\mathcal{W}}

\newcommand{\ZZ}{\mathbb{Z}}

\newcommand{\VOA}{{vertex operator algebra}}

\newcommand{\VOsA}{{vertex operator subalgebra}}
\usetikzlibrary{knots}

\usepackage{xparse}
\NewDocumentCommand\render{sg}{%
	\IfBooleanTF#1%
	{#2}  %     If a star is seen
	{}    %     If no star is seen
}
\allowdisplaybreaks

\begin{document}
	
\title{Cosets of the $\mathcal {W}^k(\mathfrak{sl}_4, f_{\text {subreg}})$-algebra}	
\date{}
\maketitle
\abstract{
Let $\mathcal {W}^k(\mathfrak{sl}_4, f_{\text {subreg}})$ be the universal $\mathcal{W}$-algebra associated to $\mathfrak{sl}_4$ with its subregular nilpotent element, and let $\mathcal {W}_k(\mathfrak{sl}_4, f_{\text {subreg}})$ be its simple quotient. There is a Heisenberg subalgebra $\mathcal{H}$, and we denote by $\mathcal{C}^k$ the coset $\text{Com}(\mathcal{H}, \mathcal {W}^k(\mathfrak{sl}_4, f_{\text {subreg}}))$, and by $\mathcal{C}_k$ its simple quotient. We show that for $k=-4+(m+4)/3$ where $m$ is an integer greater than $2$ and $m+1$ is coprime to $3$, $\mathcal{C}_k$ is isomorphic to a rational, regular $\mathcal W$-algebra $\mathcal{W}(\mathfrak{sl}_m, f_{\text{reg}})$. In particular, $\mathcal{W}_k(\mathfrak{sl}_4, f_{\text {subreg}})$ is a simple current extension of the tensor product of $\mathcal{W}(\mathfrak{sl}_m, f_{\text{reg}})$ with a rank one lattice vertex operator algebra, and hence is rational. }

\setcounter{tocdepth}{2}
\setcounter{secnumdepth}{4}
%\tableofcontents

%\onehalfspacing

\allowdisplaybreaks

\section{Introduction}

Given a \VOA{} $V$ and a \VOsA{} $W$ the subalgebra of $V$ that commutes with $W$, $C=\text{Com}(W, V)$ is called a coset \VOA{} of $V$. This was introduced by Frenkel and Zhu in \cite{FZ}, generalizing earlier constructions in \cite{KP} and \cite{GKO}, where it was used to construct the unitary discrete series representations of the Virasoro algebra. It is widely believed that if both $V$ and $W$ satisfy certain nice properties, then so does $C$. For example if $V$ and $W$ are both rational or $C_2$-cofinite then one expects $C$ to be rational or $C_2$-cofinite as well. However, general results are very difficult to obtain.

For $n\geq 4$, let $\mathcal{W}^k(\mathfrak{sl}_n, f_{\text{subreg}})$ denote the $\mathcal{W}$-algebra at level $k$ associated to $\mathfrak{sl}_n$ with its subregular nilpotent element \cite{CM,KRW}, and let $\mathcal{W}_k(\mathfrak{sl}_n, f_{\text{subreg}})$ be the simple quotient. It was recently shown by Genra \cite{Gen} that $\mathcal{W}^k(\mathfrak{sl}_n, f_{\text{subreg}})$ coincides with the Feigin-Semikhatov algebra $\mathcal{W}^{(2)}_n$ \cite{FS}, and is strongly generated by $n+1$ fields of conformal weights $1, 2, \dots, n-1, n/2, n/2$. Note that $\mathcal{W}^{(2)}_n$ is well defined for $n=2$ and $n=3$; $\mathcal{W}^{(2)}_2$ coincides with the affine \VOA{} $V^k(\mathfrak{sl}_2)$, and $\mathcal{W}^{(2)}_3$ coincides with the Bershadsky-Polyakov algebra \cite{Ber,Pol}.

The weight one field of $\mathcal {W}^k(\mathfrak{sl}_n, f_{\text {subreg}})$ generates a Heisenberg algebra $\mathcal{H}$, and we are interested in the coset 
$$\mathcal{C}^k = \text{Com}(\mathcal{H}, \mathcal {W}^k(\mathfrak{sl}_n, f_{\text {subreg}}))$$ for generic values of $k$, as well as the simple quotient $$\mathcal{C}_k = \text{Com}(\mathcal{H}, \mathcal {W}_k(\mathfrak{sl}_n, f_{\text {subreg}}))$$ at certain special values of $k$. It was conjectured in the physics literature over 20 years ago \cite{BEHHH} that there is a sequence of levels $k$ where $\mathcal{C}_k$ is isomorphic to a rational $\mathcal{W}(\mathfrak{sl}_m, f_{\text{reg}})$-algebra, where $m$ depends linearly on $k$. In the cases $n=2$ and $n=3$, where $\mathcal{W}_k(\mathfrak{sl}_n, f_{\text {subreg}})$ is replaced by the simple affine algebra $L_k(\mathfrak{sl}_2)$ and the simple Bershadsky-Polyakov algebra, respectively, this conjecture was proven recently in \cite{ALY} and \cite{ACLI}. We remark that subregular $\mathcal{W}$-algebras of type A have recently become important due to there r\^ole in four-dimensional supersymmetric gauge theories as chiral algebras of Argyres-Douglas theories \cite{BN,C,CS}. Interestingly, these are exactly those levels where the $\mathcal{W}$-algebra has the logarithmic singlet \VOA{} \cite{Kau, AM} as Heisenberg coset \cite{CRW}.

In the case $n=4$, the above conjecture states that for $k=-4+(m+4)/3$, $\mathcal {C}_k$ is isomorphic to a rational $\mathcal{W}(\mathfrak{sl}_m, f_{\text{reg}})$-algebra. Here $m$ is an integer greater than $2$ such that $m+1$ is coprime to $3$. Our main result is a proof of this conjecture. In fact, we prove the stronger statement that for these values of $k$, $\mathcal {W}_k(\mathfrak{sl}_4, f_{\text {subreg}})$ is a simple current extension of $V_L \otimes \mathcal{W}(\mathfrak{sl}_m, f_{\text{reg}})$, where $V_L$ is a certain rank one lattice VOA. As a corollary, we obtain the $C_2$-cofiniteness and rationality of $\mathcal {W}_k(\mathfrak{sl}_4, f_{\text {subreg}})$ for the above values of $k$. Additionally, we show that $\mathcal{C}^k$ is of type $\mathcal{W}(2,3,4,5,6,7,8,9)$ for all values of $k$ except for $k=-2, -5/2, -8/3$. In particular, this strong generating set works for the simple quotient $\mathcal{C}_k$ for the above values of $k$, so this family of rational $\mathcal{W}(\mathfrak{sl}_m, f_{\text{reg}})$-algebras has the following {\it uniform truncation property}. For $m\geq 9$ they are all of type $\cW(2,3,4,5,6,7,8,9)$, even though the universal $\mathcal{W}(\mathfrak{sl}_m, f_{\text{reg}})$-algebra is of type $\mathcal{W}(2,3,\dots, m)$. This happens because a singular vector of weight 10 in the universal algebra gives rise to decoupling relations in the simple quotient expressing the generators of weights $10, 11,\dots, m$ as normally ordered polynomials in the ones up to weight $9$.

Here is a brief sketch of the proof of our result.
\begin{enumerate}

\item Since $\mathcal {C}^k \otimes \mathcal H \cong \mathcal {W}^k(\mathfrak{sl}_4, f_{\text{subreg}})^{U(1)}$, studying $\mathcal {C}^k$ is equivalent to studying the $U(1)$-orbifold $\mathcal{W}^k(\mathfrak{sl}_4, f_{\text {subreg}})^{U(1)}$. We start with an obvious infinite set of generators coming from classical invariant theory in weights $1,2,3,\dots$, and we compute normally ordered relations among these generators starting in weight $10$. We show that for generic values of $k$, all generators in weights $w \geq 10$ can be eliminated using these relations. This shows that $\mathcal {W}^k(\mathfrak{sl}_4, f_{\text {subreg}})^{U(1)}$ is of type $\mathcal{W}(1,2,3,4,5,6,7,8,9)$, and hence that $\mathcal{C}^k$ is of type $\mathcal{W}(2,3,4,5,6,7,8,9)$ for generic $k$.

\item Coefficients in the normally ordered relations are polynomial functions in $k$, and zeros of coefficients correspond to non-generic values of $k$. It turns out that this only happens for $k=-2, -5/2, -8/3$, so these are the only nongeneric values.

\item We prove the conjecture of \cite{BEHHH} by showing that a certain rational $\mathcal{W}(\mathfrak{sl}_m, f_{\text{reg}})$-algebra times a rank one lattice \VOA{} allows for a simple current extension whose OPE algebra coincides with the one of $\mathcal {W}^k(\mathfrak{sl}_4, f_{\text {subreg}})$. The main ingredient is that the Jacobi identity implies that the full OPE algebra is uniquely determined by a small amount of data which is easily shown to coincide in both algebras. Therefore the two \VOA s must be isomorphic. 
\end{enumerate}

This paper is part of a broader program of the authors to study {\it deformable familes} of \VOA s, i.e. \VOA s that depend continuously on one or more parameters. Examples include universal affine vertex algebras and $\mathcal{W}$-algebras, where the parameter is the level $k$, as well as orbifolds and cosets of these algebras. In many situations, the question of finding a minimal strong generating set for an orbifold or coset of such a deformable family can be decided by passing to the limit as $k$ approaches infinity, which is often an orbifold of a free field algebra \cite{CLI}. The structure of orbifolds of free field algebras can then be determined using ideas from classical invariant theory \cite{CLIII, LI, LII, LIII, LIV}.

Besides our pure interest in \VOA{} invariant theory, our findings have quite some impact on important questions we are interested in.
It is widely believed that $\cW$-algebras that are certain quantum Hamiltonian reductions of affine \VOA s, can also be realized as coset algebras. The most famous example is surely the regular $\cW$-algebra of a simply-laced Lie algebra $\mathfrak g$ as a coset of the affine \VOA{} $V^{k+1}(\mathfrak g)$ inside $V^{k}(\mathfrak g) \otimes L_{1}(\mathfrak g)$. Jointly with Tomoyuki Arakawa we are able to prove this coset realization \cite{ACLII}. 
Other examples are the coset of $V^{k+1}(\mathfrak{so}_n)$ inside   $V^{k}(\mathfrak{so}_{n+1})\otimes \mathcal F(n)$ with $\mathcal F(n)$ the free field algebra of $n$ free fermions. This is believed to be a regular $\cW$-algebra of type $\mathfrak{osp}(2n+1|2n)$ and Lemma 7.17 of \cite{CLI} tells us that this conjecture is indeed consistent with minimal strong generating sets. These families of $\cW$-algebras all carry an action of the $N=1$ super Virasoro algebra. The $N=2$ super Virasoro case corresponds to the cosets of $V^{k+1}(\mathfrak{sl}_n)$ inside $V^{k}(\mathfrak{sl}_{n+1})\otimes \mathcal F(2n)$ and the expected $\cW$-algebra is a regular $\cW$-algebra of type $\mathfrak{sl}(n+1|n)$.
Lemma 7.12 of \cite{CLI} confirms this on the level of minimal strong generating sets. It is a rather ambitious future goal to indeed prove these conjectures. Correctness of these conjectures is the starting assumption for the (super) higher spin gravity on AdS$_3$ to two-dimensional (super) conformal field theory correspondence of \cite{GG, CHR1, CHR2}. 

There are many more conjectures emerging from deformable families of \VOA s that come from the physics of four-dimensional supersymmetric gauge theories \cite{CGai}. These involve cosets of affine \VOA s inside certain $\cW$-algebras. Besides their apparent importance in physics \cite{Gaiotto}, they also relate to the quantum geometric Langlands correspondence \cite{CGai,AFO}. Here the starting point is a series of conjectural deformable families of \VOA s extending the tensor product of two affine \VOA s with Langlands dual Lie algebras; see the introduction of \cite{CGai} for a list of these conjectures.

\section{Vertex algebras}
In this section, we define vertex algebras, which have been discussed from various points of view in the literature (see for example \cite{Bor,FLM,K}). We will follow the formalism developed in \cite{LZ} and partly in \cite{LiI}. Let $V=V_0\oplus V_1$ be a super vector space over $\mathbb{C}$, and let $z,w$ be formal variables. By $\text{QO}(V)$, we mean the space of linear maps $$V\rightarrow V((z))=\{\sum_{n\in\mathbb{Z}} v(n) z^{-n-1}|
v(n)\in V,\ v(n)=0\ \text{for} \ n>\!\!>0 \}.$$ Each element $a\in \text{QO}(V)$ can be represented as a power series
$$a=a(z)=\sum_{n\in\mathbb{Z}}a(n)z^{-n-1}\in \text{End}(V)[[z,z^{-1}]].$$ We assume that $a=a_0+a_1$ where $a_i:V_j\rightarrow V_{i+j}((z))$ for $i,j\in\mathbb{Z}/2\mathbb{Z}$, and we write $|a_i| = i$.

For each $n \in \mathbb{Z}$, we have a nonassociative bilinear operation on $\text{QO}(V)$, defined on homogeneous elements $a$ and $b$ by
$$ a(w)_{(n)}b(w)=\text{Res}_z a(z)b(w)\ \iota_{|z|>|w|}(z-w)^n- (-1)^{|a||b|}\text{Res}_z b(w)a(z)\ \iota_{|w|>|z|}(z-w)^n.$$
Here $\iota_{|z|>|w|}f(z,w)\in\mathbb{C}[[z,z^{-1},w,w^{-1}]]$ denotes the power series expansion of a rational function $f$ in the region $|z|>|w|$. For $a,b\in \text{QO}(V)$, we have the following identity of power series known as the {\it operator product expansion} (OPE) formula.
 \begin{equation}\label{opeform} a(z)b(w)=\sum_{n\geq 0}a(w)_{(n)} b(w)\ (z-w)^{-n-1}+:a(z)b(w):. \end{equation}
Here $:a(z)b(w):\ =a(z)_-b(w)\ +\ (-1)^{|a||b|} b(w)a(z)_+$, where $a(z)_-=\sum_{n<0}a(n)z^{-n-1}$ and $a(z)_+=\sum_{n\geq 0}a(n)z^{-n-1}$. Often, \eqref{opeform} is written as
$$a(z)b(w)\sim\sum_{n\geq 0}a(w)_{(n)} b(w)\ (z-w)^{-n-1},$$ where $\sim$ means equal modulo the term $:a(z)b(w):$, which is regular at $z=w$. 

Note that $:a(w)b(w):$ is a well-defined element of $\text{QO}(V)$. It is called the {\it Wick product} or {\it normally ordered product} of $a$ and $b$, and it
coincides with $a_{(-1)}b$. For $n\geq 1$ we have
$$ n!\ a(z)_{(-n-1)} b(z)=\ :(\partial^n a(z))b(z):,\qquad \partial = \frac{d}{dz}.$$
For $a_1(z),\dots ,a_k(z)\in \text{QO}(V)$, the $k$-fold iterated Wick product is defined inductively by
\begin{equation}\label{iteratedwick} :a_1(z)a_2(z)\cdots a_k(z):\ =\ :a_1(z)b(z):,\qquad b(z)=\ :a_2(z)\cdots a_k(z):.\end{equation}
We often omit the formal variable $z$ when no confusion can arise.

A subspace $\cA\subset \text{QO}(V)$ containing $1$ which is closed under all the above products will be called a {\it quantum operator algebra} (QOA). We say that $a,b\in \text{QO}(V)$ are {\it local} if $$(z-w)^N [a(z),b(w)]=0$$ for some $N\geq 0$. Here $[,]$ denotes the super bracket. This condition implies that $a_{(n)}b = 0$ for $n\geq N$, so (\ref{opeform}) becomes a finite sum. Finally, a {\it vertex algebra} will be a QOA whose elements are pairwise local. This notion is well known to be equivalent to the notion of a vertex algebra in the sense of \cite{FLM}. 

A vertex algebra $\cA$ is said to be {\it generated} by a subset $S=\{a_i|\ i\in I\}$ if $\cA$ is spanned by words in the letters $a_i$, and all products, for $i\in I$ and $n\in\mathbb{Z}$. We say that $S$ {\it strongly generates} $\cA$ if $\cA$ is spanned by words in the letters $a_i$, and all products for $n<0$. Equivalently, $\cA$ is spanned by $$\{ :\partial^{k_1} a_{i_1}\cdots \partial^{k_m} a_{i_m}:| \ i_1,\dots,i_m \in I,\ k_1,\dots,k_m \geq 0\}.$$ As a matter of notation, we say that a vertex algebra $\cA$ is of type $$\cW(d_1,\dots, d_r)$$ if it has a minimal strong generating set consisting of one field in each weight $d_1,\dots, d_r$.

Given fields $a,b,c$ in any vertex algebra $\mathcal{V}$, and integers $m,n \geq 0$, the following identities are known as {\it Jacobi relations} of type $(a,b,c)$. 
\begin{equation} \label{jacobi} a_{(r)}(b_{(s)} c) = (-1)^{|a||b|} b_{(s)} (a_{(r)}c) + \sum_{i =0}^r \binom{r}{i} (a_{(i)}b)_{(r+s - i)} c.\end{equation} For a fixed choice of fields $a,b,c$, these identites are nontrivial for only finitely many integers $m,n$.

Given a vertex algebra $\mathcal{V}$ and a vertex subalgebra $\mathcal{A} \subset \mathcal{V}$, the coset (or commutant) of $\mathcal{A}$ in $\mathcal{V}$, denoted by $\text{Com}(\mathcal{A},\mathcal{V})$, is the subalgebra of elements $v\in\mathcal{V}$ such that $[a(z),v(w)] = 0$ for all $a\in\mathcal{A}$. Equivalently, $v\in \text{Com}(\mathcal{A},\mathcal{V})$ if and only if $a_{(n)} v = 0$ for all $a\in\mathcal{A}$ and $n\geq 0$.

\section{The algebra $\cW^k(\mathfrak{sl}_4, f_{\text{subreg}})$}
Let $\cW^k(\mathfrak{sl}_4, f_{\text{subreg}})$ denote that $\cW$-algebra at level $k$ associated to $\mathfrak{sl}_4$ with its subregular nilpotent element $f_{\text{subreg}}$. By a theorem of Genra \cite{Gen}, it coincides with the algebra $\cW^{(2)}_4$ \cite{FS} in the family $\cW^{(2)}_{n}$ constructed by Feigin and Semikhatov \cite{FS}. It is freely generated by fields $J, T, W, G^{\pm}$ of weights $1,2,3,2,2$, respectively, satisfying the following OPEs.

\begin{equation} \label{OPE:FS4a} \begin{split} T(z) T(w) &\sim - \frac{(8 + 3 k) (17 + 8 k)}{2 (4 + k)} (z-w)^{-4} + 2T(w) (z-w)^{-2} + \partial T(w)(z-w)^{-1},\\ T(z) J(w) &\sim J(w)(z-w)^{-2} + \partial J(w)(z-w)^{-1},
\\T(z) W(w) &\sim 3W(w)(z-w)^{-2} + \partial W(w)(z-w)^{-1},
\\ T(z) G^{\pm} (w) &\sim 2G^{\pm} (w)(z-w)^{-2} + \partial G^{\pm} (w)(z-w)^{-1},
\\ J(z) J(w) &\sim (2 + \frac{3 k}{4})(z-w)^{-2},
\\ J(z) G^{\pm} (w) &\sim \pm G^{\pm}(w)(z-w)^{-1},\end{split} \end{equation}

\begin{equation} \label{OPE:FS4b} \begin{split} W(z) G^{\pm} (w) &\sim \pm \frac{2 (4 + k) (7 + 3 k) (16 + 5 k)}{(8 + 3 k)^2} G^{\pm}(w)(z-w)^{-3}  \\ + &
\bigg(\pm \frac{3 (4 + k) (16 + 5 k)}{2 (8 + 3 k)} \partial G^{\pm}   
 - \frac{6 (4 + k) (16 + 5 k)}{(8 + 3 k)^2} :J G^{\pm}: \bigg)(w)(z-w)^{-2} 
\\ &+ \bigg(-\frac{8 (3 + k) (4 + k)}{(2 + k) (8 + 3 k)} : J \partial G^{\pm}:   - \frac{4 (4 + k) (16 + 15 k + 3 k^2)}{(2 + k) (8 + 3 k)^2} :(\partial J) G^{\pm}: \\ & \pm \frac{(3 + k) (4 + k)}{2 + k} \partial^2 G^{\pm}  \mp \frac{2 (4 + k)^2}{(2 + k) (8 + 3 k)} :TG^{\pm}: \\ & \pm  \frac{4 (4 + k) (16 + 5 k)}{(2 + k) (8 + 3 k)^2} :JJ G^{\pm}:    \bigg)(w)(z-w)^{-1} ,\end{split} \end{equation}

\begin{equation} \label{OPE:FS4c} \begin{split} G^+(z) G^-(w)  &\sim  (2 + k) (5 + 2 k) (8 + 3 k) (z-w)^{-4} +  4 (2 + k) (5 + 2 k) J(w)(z-w)^{-3}
\\ &+ \bigg(-(2 + k) (4 + k) T + 6 (2 + k) :JJ: +2 (2 + k) (5 + 2 k) \partial J\bigg)(w)(z-w)^{-2} \\&+ \bigg( (k+2) W + \frac{8 (2 + k) (32 + 11 k)}{3 (8 + 3 k)^2} :JJJ:   -\frac{4 (2 + k) (4 + k)}{8 + 3 k} :TJ: +6 (2 + k) :(\partial J)J: \\& -\frac{1}{2} (2 + k) (4 + k) \partial T 
 + \frac{4 (2 + k) (26 + 17 k + 3 k^2)}{3 (8 + 3 k)} \partial^2 J\bigg)(w)(z-w)^{-1}.\end{split} \end{equation}
 
\begin{equation} \label{OPE:FS4d} W(z) W(w) \sim \frac{2 (k + 4) (2 k + 5) (3 k + 7) (5 k + 16)}{3 k + 8} (z-w)^{-6} + \cdots, \end{equation}
The remaining terms in the OPE of $W(z)W(w)$ have been omitted but can be found in the paper \cite{FS}. As in \cite{ACLI}, a {\it weak increasing filtration} on a vertex algebra $\cA$ is a $\mathbb{Z}_{\geq 0}$-filtration
\begin{equation}\label{weak} \cA_{(0)}\subset\cA_{(1)}\subset\cA_{(2)}\subset \cdots,\qquad \cA = \bigcup_{d\geq 0}
\cA_{(d)}\end{equation} such that for $a\in \cA_{(r)}$, $b\in\cA_{(s)}$, we have
\begin{equation} a\circ_n b \in \cA_{(r+s)},\qquad n\in \mathbb{Z}.\end{equation}
Then the associated graded algebra $\text{gr}(\cA) = \bigoplus_{d\geq 0}\cA_{(d)}/\cA_{(d-1)}$ is a vertex algebra, and a strong generating set for $\text{gr}(\cA)$ lifts to a strong generating set for $\cA$; see Lemma 4.1 of \cite{ACLI}. We define a filtration \begin{equation} \label{filt} \cW^k(\mathfrak{sl}_4, f_{\text{subreg}})_{(0)} \subset \cW^k(\mathfrak{sl}_4, f_{\text{subreg}})_{(1)} \subset \cW^k(\mathfrak{sl}_4, f_{\text{subreg}})_{(2)}  \subset \cdots \end{equation} on $\cW^k(\mathfrak{sl}_4, f_{\text{subreg}})$ as follows: $\cW^k(\mathfrak{sl}_4, f_{\text{subreg}})_{(-1)} = \{0\}$, and $\cW^k(\mathfrak{sl}_4, f_{\text{subreg}})_{(r)}$ is spanned by iterated Wick products of the generators $J,T,W,G^{\pm}$ and their derivatives, such that at most $r$ of the fields $W,G^{\pm}$ and their derivatives appear. 
It is clear from the defining OPE relations that this is a weak increasing filtration.

\section{The $U(1)$-orbifold of $\cW^k(\mathfrak{sl}_4, f_{\text{subreg}})^{U(1)}$}

The action of the zero mode $J_0$ integrates to a $U(1)$-action on $\cW^k(\mathfrak{sl}_4, f_{\text{subreg}})$, and the orbifold $\cW^k(\mathfrak{sl}_4, f_{\text{subreg}})^{U(1)}$ is just the kernel of $J_0$. Since $J,T,W$ lie in $(\cW^k(\mathfrak{sl}_4, f_{\text{subreg}})^{U(1)}$ and $J_0(G^{\pm}) = \pm G^{\pm}$, it is immediate that $\cW^k(\mathfrak{sl}_4, f_{\text{subreg}})^{U(1)}$ is spanned by all normally ordered monomials of the form 
\begin{equation} \label{standardmonomial} \begin{split} :(\partial^{a_1} T) \cdots (\partial^{a_i} T) (\partial^{b_1} J )\cdots (\partial^{b_j} J) (\partial^{c_1} W) \cdots (\partial^{c_r} W) (\partial^{d_1} G^+) \\ \cdots (\partial^{d_s} G^+)( \partial^{e_1} G^-) \cdots (\partial^{e_s} G^-):,\end{split} \end{equation} where $i,j,r,s\geq 0$ and $a_1\geq \cdots \geq a_i \geq 0$, $b_1\geq \cdots \geq b_j \geq 0$, $c_1\geq \cdots \geq c_r \geq 0$, $d_1\geq \cdots \geq d_s \geq 0$, and $e_1\geq \cdots \geq e_s \geq 0$. 

The filtration \eqref{filt} on $\cW^k(\mathfrak{sl}_4, f_{\text{subreg}})$ restricts to a weak increasing filtration on $\cW^k(\mathfrak{sl}_4, f_{\text{subreg}})^{U(1)}$ where $$(\cW^k(\mathfrak{sl}_4, f_{\text{subreg}})^{U(1)})_{(r)} = \cW^k(\mathfrak{sl}_4, f_{\text{subreg}})^{U(1)} \cap \cW^k(\mathfrak{sl}_4, f_{\text{subreg}})_{(r)}.$$ Define
$$U_{i,j} =\ :\partial^i G^+ \partial^j G^-: ,$$ which lies in $(\cW^k(\mathfrak{sl}_4, f_{\text{subreg}})^{U(1)})_{(2)}$ and has weight $i+j+4$. By the same argument as the proof of Lemma 5.1 of \cite{ACLI}, we have

\begin{lemma} $\cW^k(\mathfrak{sl}_4, f_{\text{subreg}})^{U(1)}$ is strongly generated as a vertex algebra by
\begin{equation} \label{gen} \{J, T, W, U_{0,m}|\ m\geq 0\}.\end{equation}
\end{lemma}

However, $\cW^k(\mathfrak{sl}_4, f_{\text{subreg}})^{U(1)}$ is not freely generated by \eqref{gen}. There is a relation of weight $10$ of the form
$$\frac{(2 + k) (5 + 2 k) (8 + 3 k)}{360} U_{0,6} = \ :U_{0,0} U_{1,1}: - :U_{0,1} U_{1,0}: + \cdots,$$ where the remaining terms are normally ordered monomials in $T,J,W,U_{0,i}$ and their derivatives, for $i\leq 5$. We can regard $:U_{0,0} U_{1,1}: - :U_{0,1} U_{1,0}:$ as the analogue of a classical relation which does not vanish due to the nonassociativity of the Wick product, and the remaining terms provide the necessary corrections to make it a genuine relation. This relation is unique up to scalar multiples, and the coefficient of $U_{0,6}$ is canonical in the sense that it does not depend on any choices of normal ordering in the expression on the right side. In particular, we see that $U_{0,6}$ decouples for all $k \neq -2, -5/2, -8/3$.

Similarly, for all $n>1$ we have relations
$$\frac{n (9 + n)(2 + k) (5 + 2 k) (8 + 3 k)}{120 (4 + n) (5 + n)} U_{0,n+5} = \ :U_{0,0} U_{1,n}: - :U_{0,n} U_{1,0}: + \cdots$$ where the remaining terms are normally ordered monomials in $T,J,W,U_{0,i}$ and their derivatives, for $i\leq 5$. The proof is similar to the proof of Theorem 5.4 of \cite{ACLI}. Again, the coefficient of $U_{0,n+5}$ is canonical, and this shows that $U_{0,n+5}$ can be decoupled for all $n>1$ whenever $k \neq -2, -5/2, -8/3$. We obtain

\begin{theo} For all $k \neq -2, -5/2, -8/3$, $\cW^k(\mathfrak{sl}_4, f_{\text{subreg}})^{U(1)}$ has a minimal strong generating set $$\{J,T,W,U_{0,i}|\ i\leq 5\},$$ and in particular is of type $\cW(1,2,3,4,5,6,7,8,9)$.
\end{theo}

\section{The Heisenberg coset of $\cW^k(\mathfrak{sl}_4, f_{\text{subreg}})$}

Let $\cH \subset \cW^k(\mathfrak{sl}_4, f_{\text{subreg}})$ denote the copy of the Heisenberg vertex algebra generated by $J$, and let $\cC^{k}$ denote the commutant $\text{Com}(\cH, \cW^k(\mathfrak{sl}_4, f_{\text{subreg}}))$. Note that $$\cW^k(\mathfrak{sl}_4, f_{\text{subreg}})^{U(1)}  \cong \cH \otimes \cC^{k}$$ and $\cC^{k}$ has a Virasoro element $$T^{\cC} = T - \frac{2}{8 + 3 k} :JJ:$$  of central charge $$c = -\frac{4 (5 + 2 k) (7 + 3 k)}{4 + k}.$$ Also, it is clear from the OPE algebra that $W\in \cC^k$. By a straightforward computer calculation, we obtain

\begin{theo} \label{truncation} For $0\leq i \leq 5$, and $k \neq -2, -5/2, -8/3$, there exist correction terms $\omega_i \in \cW^k(\mathfrak{sl}_4, f_{\text{subreg}})^{U(1)} $ such that $U^{\cC}_i = U_{0,i} + \omega_i$ lies in $\cC^{k}$. Therefore $\cC^{k}$ has a minimal strong generating set $\{T^{\cC}, W, U^{\cC}_i|\ 0\leq i \leq 5\}$, and is therefore of type $\cW(2,3,4,5,6,7,8,9)$.  \end{theo}

Next, let $ \cW_k(\mathfrak{sl}_4, f_{\text{subreg}})$ denote the simple quotient of $ \cW^k(\mathfrak{sl}_4, f_{\text{subreg}})$ by its maximal proper ideal graded by conformal weight, and let $\cC_k =  \text{Com}(\cH, \cW^k(\mathfrak{sl}_4, f_{\text{subreg}}))$. Evidently we have a surjective map $$\cC^k \rightarrow \cC_k,$$ so for $k \neq -2, -5/2, -8/3$, $\cC_k$ is strongly generated by the fields above.

\section{Simple current extensions and $\cW_{\ell}(\mathfrak{sl}_n, f_{\text{reg}})$}

Vertex operator algebra extensions of a given vertex algebra $V$ can be efficiently studied using commutative, associative algebras with injective unit in the representation category of $V$. This has been developed in \cite{KO, HKL, CKM} and especially structure about parafermionic cosets, i.e. cosets by a Heisenberg or lattice vertex algebra, has been derived in \cite{CKL, CKLR, CKM}. Here we use these ideas to construct simple current extensions of rational, regular $\cW$-algebras of type $A$ tensored with certain lattice vertex operator algebras. Recall that a simple current is an invertible object in the tensor category of the vertex operator algebra.

Let $n, r$ be in $\ZZ_{>1}$ such that $n+1$ and $n+r$ are coprime (so that especially $nr$ is even) and define
\[
\cW(n, r) := \cW_{\ell}(\mathfrak{sl}_n, f_{\text{reg}}), \qquad \ell+n = \frac{n+r}{n+1}.
\] By \cite{Ar1}, $\cW(n, r)$ is rational and $C_2$-cofinite. Let $L=\sqrt{nr}\ZZ$ and $V_L$ the lattice vertex operator algebra of $L$. Modules and their fusion rules for $\cW(n, r)$ are essentially known due to \cite{FKW, AvE}. Modules are parameterized by modules of $L_{r}\left(\mathfrak{sl}_n\right)$, i.e. by integrable positive weights of $\widehat{\mathfrak{sl}}_n$ at level $r$. Fusion rules (Theorem 4.3, Proposition 4.3 of \cite{FKW} together with Corollary 8.4 of \cite{AvE}) imply that the group of simple currents is $\ZZ/n\ZZ$ and these simple currents correspond to the modules $\mathbb L_{r\omega_i}$ with $\omega_i$ the fundamental weights of $\mathfrak{sl}_n$. The question of extending a given regular vertex algebra by a group of simple currents to a larger one is entirely decided by conformal dimension and quantum dimension of the involved simple currents. One gets a vertex operator superalgebra if and only if conformal dimensions of a set of generators of the group of simple currents are in $\frac{1}{2}\ZZ$. Moreover the quantum dimension of generators of the group of simple currents decide whether this is even a vertex operator algebra. See \cite{CKL} for details. 

By the quantum dimension of a module $M$ we mean the categorical dimension of $M$. By Verlinde's formula \cite{H1, H2} one has
\[
\text{qdim}(M) = \frac{S_{M, V}}{S_{V, V}}
\]
with $S$-matrix of the modular transformation of torus one-point functions $\mathrm{ch}[M](v, \tau):= \mathrm{tr}_M(o(v)q^{L_0-c/24})$ ($v$ in $V$ of conformal weight $k$ and $o(v)$ the zero-mode of $v$)
\[
\mathrm{ch}[{M}](v,-1/\tau) = \tau^k\sum_{N}  S_{M, N}\,\mathrm{ch}[{N}](v,\tau)\,.
\]
The sum here is over all inequivalent modules of $V$. See \cite{CG} for a review on modular and categorical aspects of vertex algebras. 

The quantum dimension and conformal dimension of $\mathbb L_{r\omega_1}$ are now easily computed using the recent results of van Ekeren and Arakawa \cite{AvE}:
\begin{equation}
\begin{split}
\text{qdim} \left(\mathbb L_{r\omega_1} \right) &= \frac{S_{r\omega_1, 0}}{S_{0, 0}} = e^{2\pi i r\left(\omega_1, \rho\right)}\text{qdim} \left(L_{r\omega_1} \right) = e^{2\pi i r\left(\omega_1, \rho\right)}\\ &= e^{2\pi i r\frac{n-1}{2}}=(-1)^{r(n-1)}=(-1)^r 
\end{split}
\end{equation}
Here $\text{qdim} \left(L_{r\omega_1} \right)$ is the quantum dimension of the $L_{r}\left(\mathfrak{sl}_n\right)$ module $L_{r\omega_1}$.  We firstly used that the modular $S$-matrices of $\cW(n, r)$ and of $L_r(\mathfrak{sl}_n)$ only differ by the factor $e^{2\pi i r\left(\omega_1, \rho\right)}$ with $\rho$ the Weyl vector of $\mathfrak{sl}_n$. Secondly we used that $L_r(\mathfrak{sl}_n)$ is unitary and hence all quantum dimensions are positive and so every simple current of $L_r(\mathfrak{sl}_n)$ must have quantum dimension one. Finally, in the last equality we used that $nr$ is even. 
The conformal dimension is
\[
\Delta(\mathbb L_{r\omega_1}) = \frac{(n+1)}{2(n+r)}\left( r\omega_1, r\omega_1+2\rho\right) -(\omega_1, \rho) = \frac{(n-1)r}{2n}
\]
since $\omega_1^2= \frac{n-1}{n}$ and $(\omega_1, \rho)=\frac{n-1}{2}$. We denote by $V_{L+\gamma}$ the $V_L$-module corresponding to the coset $L+\gamma$ of $L$ in the dual lattilce $L'=\frac{1}{\sqrt{nr}}\ZZ$. Then $V_{L+\frac{r}{\sqrt{rn}}}$ has conformal dimension $\frac{r}{2n}$ and quantum dimension one since $V_L$ is unitary. It follows from \cite{CKL} (see the Theorems listed in the introduction of that work) that 
\begin{equation}
A(n, r) \cong  \bigoplus_{s=0}^{n-1} V_{L+\frac{rs}{\sqrt{rn}}} \otimes \underbrace{\mathbb L_{r\omega_1}\boxtimes_{\cW(n, r)} \dots \boxtimes_{\cW(n, r)}  \mathbb L_{r\omega_1}}_{s-\text{times}}
\end{equation}
is a vertex operator algebra extending $V_L \otimes \cW(n, r)$. If $r$ is even, this is a $\ZZ$-graded vertex operator algebra, while for odd $r$ it is only $\frac{1}{2}\ZZ$-graded. 
The subspace of lowest conformal weight in each of the $\mathbb L_{r\omega_1}\boxtimes_{\cW(n, r)} \dots \boxtimes_{\cW(n, r)}  \mathbb L_{r\omega_1}$ is one-dimensional, and we denote the corresponding vertex operators by $X_s$. By Proposition 4.1 of \cite{CKL} the OPE of $X_s$ and $X_{n-s}$ has a non-zero multiple of the identity as leading term. Without loss of generality, we may rescale $X_1$ and $X_{n-1}$ so that
\begin{equation} \label{norm1} 
X_1(z)X_{n-1}(w) \sim  \prod_{i=1}^{n-1}(i(k+n-1)-1) (z-w)^{-r} + \dots.
\end{equation}
Let $J$ be the Heisenberg field of $V_L$ and we normalize it such that
\begin{equation} \label{norm2} 
J(z)J(w) \sim \bigg(\frac{(n-1)k}{n} +n-2\bigg) (z-w)^{-2}.\end{equation}
Then we have
\begin{equation} \label{norm3} 
J(z)X_1(w) \sim X_1(w)(z-w)^{-1}, \qquad J(z)X_{n-1}(w) \sim - X_{n-1}(w) (z-w)^{-1}.
\end{equation}
Let $\widetilde A(n, r)$ be the vertex algebra generated by $X_1$ and $X_{n-1}$ under operator products
We now rephrase a physics conjecture \cite{BEHHH}, 
\begin{conj} \label{conj:main}
Let $n, r$ as above and $k$ defined by $k+r=\frac{n+r}{r-1}$. Then
\[
A(n, r) \cong \widetilde A(n, r) \cong \cW_k(\mathfrak{sl}_r, f_{\text{subreg}}).
\]
In particular, $\cW_k(\mathfrak{sl}_r, f_{\text{subreg}})$ is rational and $C_2$-cofinite.  
\end{conj}

We remark that Conjecture \ref{conj:main} is true for $r=2, 3$ by \cite{ALY}  and \cite{ACLI} and we will now prove it for $r=4$ under some extra condition on $n$. 
For this, we now assume that $n-1$ is co-prime to at least one of $n+1$ and $n+r$ so that especially $n$ even would work. Under this condition the formula for fusion rules is more explicit, and we know from the fusion rules of $\cW(n, r)$ \cite{AvE} that
\begin{equation}
\begin{split}
A(n, r) &\cong  \bigoplus_{s=0}^{n-1} V_{L+\frac{rs}{\sqrt{rn}}} \otimes \underbrace{\mathbb L_{r\omega_1}\boxtimes_{\cW(n, r)} \dots \boxtimes_{\cW(n, r)}  \mathbb L_{r\omega_1}}_{s-\text{times}}\\
&\cong \bigoplus_{s=0}^{n-1} V_{L+\frac{rs}{\sqrt{rn}}} \otimes \mathbb L_{r\omega_s}.
\end{split}
\end{equation}
The lowest conformal weight of the $s$-th summand is min$\{ \frac{sr}{2}, \frac{(n-s)r}{2}\}$ and so in this instance $\widetilde A(n, r)$ is strongly generated by $X_1, X_{n-1}$ together with the Heisenberg field $J$ and some fields of $\cW(n, r)$.

\begin{theo} \label{main} Conjecture \ref{conj:main} holds for $r = 4$ and all $n$ such that $n-1$ is co-prime to at least one of $n+1$ and $n+4$.
\end{theo}

\begin{proof} Let $L$ be the Virasoro field of $\cW(n,r)$ and let $T = L + \frac{2}{8 + 3 k}:JJ:$ be the Virasoro field of $V_L \otimes \cW(n,r)$. Also, let $W$ be the weight $3$ field of $\cW(n,r)$ which is known to generate $\cW(n,r)$. Since the OPE of $X_1(z) X_{n-1}(w)$ can be expressed in terms of $J, T,W$, the most general form is 
\begin{equation} \begin{split}X_1(z) X_{n-1} (w) & \sim  (2 + k) (5 + 2 k) (8 + 3 k) (z-w)^{-4} +  a_{1} J(w)(z-w)^{-3}
\\ &+ \bigg(a_{2} T + a_{3} :JJ: +a_{4} \partial J\bigg)(w)(z-w)^{-2} \\ &+ \bigg( a_{5} W + a_{6} :JJJ: + a_{7} :TJ: +a_{8} :(\partial J)J: + a_{9}\partial T + a_{10} \partial^2 J\bigg)(w)(z-w)^{-1},\end{split} \end{equation} where the $a_i$ are constants.
By imposing all Jacobi relations of the form $(J, X_1, X_{n-1})$ and $(T, X_1, X_{n-1})$ we obtain all the above coefficients uniquely except for $a_5$, that is, 
\begin{equation} \begin{split}\label{ope:xx} X_1(z) X_{n-1} (w) &\sim  (2 + k) (5 + 2 k) (8 + 3 k) (z-w)^{-4} +  4 (2 + k) (5 + 2 k) J(w)(z-w)^{-3}
\\ & + \bigg(-(2 + k) (4 + k) T + 6 (2 + k) :JJ: +2 (2 + k) (5 + 2 k) \partial J\bigg)(w)(z-w)^{-2} \\ &+ \bigg( a_{5} W + \frac{8 (2 + k) (32 + 11 k)}{3 (8 + 3 k)^2} :JJJ:   -\frac{4 (2 + k) (4 + k)}{8 + 3 k} :TJ: \\ &+6 (2 + k) :(\partial J)J:  -\frac{1}{2} (2 + k) (4 + k) \partial T 
\\ & + \frac{4 (2 + k) (26 + 17 k + 3 k^2)}{3 (8 + 3 k)} \partial^2 J\bigg)(w)(z-w)^{-1}.\end{split} \end{equation}

Using the OPE relations \eqref{norm3}, and the Jacobi relations of type $(X_1, X_1, X_{n-1})$, we see that $a_5 \neq 0$. Since we are free to rescale the field $W$, we may assume without loss of generality that $$a_5 = (k+2).$$ This completely determines $X_1(z)X_{n-1}(w)$. Also, since $W$ appears in $\widetilde A(n, 4)$ and generates $\cW(n,4)$ (see Proposition A.3 of \cite{ALY}), we must have $\widetilde A(n, 4) = A(n, 4)$.

Next, imposing all Jacobi relations of type $(T,W,X_1)$, $(J, W, X_1)$, $(T,W,X_{n-1})$ and $(J, W, X_{n-1})$ uniquely determines the OPEs 
\begin{equation} \label{ope:wx} W(z)X_1(w),\qquad W(z) X_{n-1}(w).\end{equation}
Finally, using \eqref{norm1}-\eqref{norm3} and \eqref{ope:xx}-\eqref{ope:wx} and imposing all Jacobi relations of type $(W, X_1, X_{n-1})$ uniquely determines the OPE of $W(z) W(w)$. In particular, these OPE relations are precisely the OPE relations in $\cW_k(\mathfrak{sl}_4, f_{\text{subreg}})$ with $X_1, X_{n-1}$ replaced by $G^+, G^-$. Since $A(n,4)$ and $\cW_k(\mathfrak{sl}_4, f_{\text{subreg}})$ are simple vertex algebras with the same strong generating set and OPE algebra, they must be isomorphic. \end{proof}

\begin{corol} Let $k$ be defined by $k+4=\frac{n+4}{3}$, and assume that $n-1$ is co-prime to at least one of $n+1$ and $n+4$. Then $\cW(n,4)$ is strongly generated by the fields in weights $2,3,4,5,6,7,8,9$ even though the universal regular $\cW$-algebra of $\mathfrak{sl}_n$ is of type $\cW(2,3,\dots, n)$.
\end{corol}

\begin{proof} This is immediate from Theorems \ref{truncation} and \ref{main} and the fact that the map $\mathcal{C}^k \rightarrow \mathcal{C}_k$ is surjective. 
\end{proof}

\bibliographystyle{alpha}

\end{document}